\documentclass{article}

\usepackage{amsmath}
\usepackage{amsthm}
\usepackage{amssymb}
\usepackage{url}
\usepackage{graphicx}
\newtheorem{thm}{Theorem}

\newtheorem{cor}{Corollary}
\newtheorem{prop}{Proposition}
\theoremstyle{definition}

\newtheorem{remark}{Remark}

\begin{document}

\title{A One-Line Proof of the Fundamental Theorem of Algebra with Newton's Method as a Consequence}
\author{Bahman Kalantari \\
Department of Computer Science, Rutgers University, Piscataway, NJ 08854\\
kalantari@cs.rutgers.edu}
\date{}
\maketitle

\begin{abstract}
Many proofs of the fundamental theorem of algebra rely on the fact that the minimum of the modulus of a complex polynomial over the complex plane is attained at some complex number. The proof then follows by arguing the minimum value is zero. This can be done by proving that at any complex number that is not a zero of the polynomial we can exhibit a direction of descent for the modulus.
In this note we present a very short and simple proof of the existence of such descent direction. In particular, our descent direction gives rise to Newton's method for solving a polynomial equation via modulus minimization and also makes the iterates definable at any critical point.
\end{abstract}


\section{Introduction}
We assume the reader has familiarity with a complex number $z=x+ iy$, $i=\sqrt{-1}$,  its conjugate $\overline z=x-iy$, its  modulus $|z|=\sqrt{z \overline z}=\sqrt{x^2+y^2}$, and Euler's formula, $e^{i \theta} =\cos \theta + i \sin \theta$.

To prove the fundamental theorem of algebra (FTA), that a nonconstant  complex polynomial $p(z)$ must have a zero, many proofs rely on the fact that the minimum of the modulus function
\begin{equation}
F(z)=|p(z)|^2=p(z) \overline{p(z)}
\end{equation}
over the complex plane is attained at a point $z_*$. This fact can be shown by first observing that  $|p(z)|$ approaches infinity as $|z|$ does. Thus the set $S = \{z : |p(z)| \leq|p(0)|\}$ is bounded. Then, by continuity of $p(z)$, $S$ is also
closed. Hence, the minimum of $F(z)$ over $S$ is attained and coincides with its minimum over the entire complex plane.  If $p(z_*)$ is nonzero, it suffices to exhibit a {\it direction of  descent} for $F(z)$ at $z_*$, i.e. a complex number $d$  such that for some positive real number $\alpha_*$ we have,
\begin{equation}
F(z_*+\alpha d) < F(z_*), \quad  \forall \alpha  \in (0, \alpha_*).
\end{equation}
This would then contradict optimality of $z_*$.
For proofs of the FTA based on a descent direction, see \cite{Fin}, \cite{BK2011}, \cite{Kon}, \cite{Lit}, \cite{Rio}. In fact, \cite{BK2011} gives a complete characterization of all descent and ascent directions.

Given a nonconstant complex polynomial $p(z)$, for any $z_0$  that is not a zero of $p(z)$, we explicitly define a direction that is a decent direction of the modulus function.  The use of this specific direction not only gives a proof of the FTA but justifies the definition of Newton's iteration for complex polynomials. It also allows defining the iterate when  $z_0$ is a critical point (i.e. $p'(z_0)=0$).

Newton's method for root finding  is traditionally studied for real polynomials and is the best known such method.  Given a real polynomial $p(x)$, and a seed $x_0$, Newton's iterations are defined as
\begin{equation} \label{eq3}
x_{j+1}=x_j- \frac{p(x_j)}{p'(x_j)}, \quad j=0, 1,...,
\end{equation}
where $p'(x)$ is the derivative of $p(x)$.  The iterate $x_{j+1}$ is undefined when $p'(x_j)=0$.  The iterate $x_{j+1}$ can be interpreted as the root of the tangent line to $p(x)$ at $x_j$.  Graphically, this property is taken as justification that when $x_0$ is close to a root of $p(x)$ the iterates converge to $\theta$. The latter can be proved analytically. In particular, $|p(x_j)|$ converges to zero. In general Newton iterates do not necessarily  decrease $|p(x)|$ monotonically, i.e. it is possible to have $|p(x_{j +1})| > |p(x_j)|$.

Cayley \cite{cayley79} is among the first to have considered Newton's iterations for a complex polynomial $p(z)$. Given $p(z)$ and a complex seed $z_0$, Newton's iterates are defined as in (\ref{eq3}), replacing $x_j$ with $z_j$. In this case too Newton's iteration $z_{j+1}$  is the solution to the linear equation, $p(z_j)+p'(z_j)(z-z_j)=0$. However, in the next section we will give another interpretation and derivation of the method in terms of modulus minimization.

\section{A Direction of Descent At Any Point, Critical or Not}

\begin{prop} \label{prop1} Let $u$ be a nonzero complex number. Given a natural number $k$, set
\begin{equation} \label{eq0}
G_k(z)=\overline u z^{k}+ u {\overline z}^k.
\end{equation}
Define the real numbers $\gamma$ and $\delta$ as
\begin{equation} \label{eq01}
\gamma=u^{(k-1)}+\overline{u}^{(k-1)}, \quad  \delta=i(u^{(k-1)}-\overline{u}^{(k-1)}).
\end{equation}
Then for any $\alpha >0$ and any real $\theta$,  we have
\begin{equation} \label{maineq}
G_k(\alpha e^{i \theta} u)=  \alpha^k |u|^2   (\gamma  \cos k\theta+   \delta \sin k \theta ).
\end{equation}
In particular,  given $\alpha >0$, we can select $\theta \in \{0, \pi/2k, \pi/k, 3\pi/2k\}$, so that  $G_k(\alpha e^{i \theta} u) <0$. Specifically,

\begin{equation} \label{eq5}
\text{Set} \quad \theta=
\begin{cases}
0, &\text{if $\gamma <0  $}; \\
\frac{\pi}{k},&\text{if $\gamma >0 $};\\
\frac{\pi}{2k},&\text{if $\delta <0 $};\\
\frac{3\pi}{2k},&\text{if $\delta >0$}.
\end{cases}
\quad \quad
\text{Then}, \quad G_k(\alpha e^{i \theta} u)=
\begin{cases}
~~\gamma \alpha^k |u|^2 <0; \\
-\gamma \alpha^k |u|^2 <0; \\
~~\delta \alpha^k |u|^2 <0; \\
-\delta \alpha^k |u|^2 <0.
\end{cases}
\end{equation}
\end{prop}

\begin{proof}  Equation (\ref{maineq}) is easily verifiable from straightforward properties of conjugation, exponentiation and Euler's formula.
Since $u \not =0$, it is easy to show $\gamma$ and $\delta$  cannot both  be zero.
Thus from (\ref{eq5}) we can choose $\theta$ so that for all $\alpha >0$,  $G_k(\alpha e^{i \theta}u)= c \alpha^k$, $c$ a negative constant.
\end{proof}

 Equation (\ref{maineq}) together with appropriate selection of $u$ and $\theta$ hold the essence of our proof of the FTA as they give rise to a descent direction for the modulus function.

\begin{thm} \label{thm1} Let $p(z)$ be a polynomial of degree $n \geq 1$. Let $F(z)=|p(z)|^2=p(z) \overline {p(z)}$. Suppose $p(z_0) \not =0$. Let $k \geq 1$ be the smallest index with $p^{(k)}(z_0) \not=0$. Let
\begin{equation} \label{udef}
u=\frac{1}{k!}p(z_0) \overline {p^{(k)}(z_0)}.
\end{equation}
Then for any $\alpha >0$ and any real $\theta$, we have
\begin{equation}\label{eq1}
F(z_0+ \alpha e^{i \theta} u)-F(z_0)=|p(z_0+ \alpha e^{i \theta} u)|^2- |p(z_0)| ^2= G_k(\alpha e^{i \theta} u) + q(\alpha),
\end{equation}
where $G_k(\alpha e^{i \theta} u)$ is as in (\ref{maineq}), and $q(\alpha)$ is a polynomial of degree $2n$ having
$\alpha^{l}$ as a factor, with $l > k$. In particular,  by selecting $\theta$ appropriately, $e^{i\theta} u$ is a descent direction for $F(z)$ at $z_0$.
\end{thm}

\begin{proof}  We  derive (\ref{eq1}) by setting $z=z_0+ \alpha e^{i \theta} u$ in the Taylor's expansion formula
\begin{equation}
p(z)=p(z_0)+ \sum_{j=k}^n \frac{p^{(j)}(z_0)}{j!}(z-z_0)^j,
\end{equation}
More specifically,  for $j=0, \dots, n$, denote $p^{(j)}(z_0)/{j!}$ by $a_j$. Thus from (\ref{udef}), $u=a_0 \overline {a_k}$.
 We have
 \begin{equation}
 F(z_0+ \alpha e^{i \theta} u)=  \bigg ( a_0+ \sum_{j=k}^n a_j \alpha^j e^{i j\theta } u^j \bigg ) \bigg ( \overline{a_0}+ \sum_{j=k}^n \overline{a_j}\alpha^j e^{-i j\theta } \overline{u}^j \bigg).
 \end{equation}
We can write the above as a real polynomial in $\alpha$ of the form $b_0+b_k \alpha^k+ b_{k+1} \alpha^{k+1}+ \dots + b_{2n} \alpha^{2n}$, where
 \begin{equation} \label{eqaa}
b_0= |a_0|^2, \quad  b_k \alpha^k= \alpha^k (\overline{u}u^k e^{ik \theta}  +  u \overline {u}^k e^{-ik \theta} ).
 \end{equation}

From (\ref{eq0}) and (\ref{eqaa}) it follows that $b_k \alpha^k$ coincides with  $G_k(\alpha e^{i\theta}u)$. In each of the remaining terms $b_j\alpha^j$, $j > k$.  Hence their sum gives a real polynomial $q(\alpha)$ of degree $2n$ claimed in (\ref{eq1}). From  Proposition \ref{prop1} we can choose $\theta$ so that for all $\alpha >0$,  $G_k(\alpha e^{i \theta}u) <0$. Since $q(\alpha)$ goes to zero faster than $\alpha^k$ does, for this $\theta$ there exists $\alpha_* >0$ so that for all $\alpha \in (0, \alpha_*)$ the right-hand-side of (\ref{eq1}) is negative.
\end{proof}

As a consequence of Theorem \ref{thm1} we have the following which justifies  Newton's method for complex polynomials.

\begin{cor} If $p'(z_0) \not =0$, $-p(z_0)/p'(z_0)$, is a descent direction for $F(z)$ at $z_0$.
\end{cor}

\begin{proof} Since $k=1$, (\ref{udef}) gives $u=p(z_0) \overline {p'(z_0)}$. Equivalently, $u={|p'(z_0)|^2 p(z_0)}/{p'(z_0)}$. Then (\ref{eq0}) gives
 $\gamma=1$, $\delta=0$. Then from  Proposition \ref{prop1} and Theorem \ref{thm1}, $e^{i \pi}u=-u$ is a descent direction.
\end{proof}

\begin{remark}  When $u$ is a nonzero real number, (\ref{eq0}) implies $\delta=0$ so that there is only one descent direction. However, when $\lambda$ and $\delta$ are both nonzero two directions of descent are implies by (\ref{eq5}).
\end{remark}

Theorem \ref{thm1} also allows defining Newton's iteration at a critical points $z_0$. Using $u$ as defined in the theorem, the Newton iterate can be defined as
\begin{equation}
z_1=z_0+e^{i \theta}\frac{p(z_0)}{k!p^{(k)}(z_0)}=z_0+e^{i \theta}\frac{|p^{(k)}(z_0)|^2}{k!} \frac{p(z_0)}{ p^{(k)}(z_0)},
\end{equation}
where $\theta$ is selected as in Proposition 1, (\ref{eq5}), and $k$ is as in Theorem \ref{thm1}. As an example suppose $p(z)=z^2-c$, $c$ a positive constant. Then  $p'(0)=0$.  Thus from (\ref{udef}) $u={p(0) \overline {p''(0)}}/{2}=-c$. Then from (\ref{eq0}) $\gamma=-2c <0$, $\delta=0$.  Thus $\theta=0$, giving $z_1=c$.

\small{

}

\bigskip

\end{document}